\documentclass[a4paper]{article}

\usepackage{a4wide}
\usepackage{amsthm}
\usepackage{psfrag}
\usepackage{graphicx}
\usepackage{epsfig}
\usepackage{amsmath}
\usepackage{amssymb}
\usepackage{amsxtra}
\usepackage{amsfonts}
\usepackage{mathrsfs}
\usepackage{color}
\usepackage{tikz}
\usepackage[compress]{cite}

\DeclareSymbolFont{bbold}{U}{bbold}{m}{n}
\DeclareSymbolFontAlphabet{\mathbbm}{bbold}

\newcommand{\diag}{\mathrm{diag}}

\newtheorem{lemma}{Lemma}
\newtheorem{remark}{Remark}
\title{\LARGE \bf Scalable computation for
  optimal control of cascade systems with constraints}
\author { Michael Cantoni, Farhad Farokhi, Eric Kerrigan, Iman Shames
  \thanks{Cantoni, Farokhi and Shames are with the Department of
    Electrical and Electronic Engineering, The University of
    Melbourne, Australia (Email:
    \texttt{\{cantoni,farhad.farokhi,iman.shames\}@unimelb.edu.au}).
    Kerrigan is with the Departments of Electrical and Electronic
    Engineering and Aeronautics, Imperial College London, U.K. (Email:
    \texttt{e.kerrigan@imperial.ac.uk}). This work is supported by the
    Australian Research Council (LP130100605) and a McKenzie
    Fellowship.}}  
\begin{document}
\maketitle \thispagestyle{empty}
\begin{abstract}
  A method is devised for numerically solving a class of
  finite-horizon optimal control problems subject to cascade linear
  discrete-time dynamics. It is assumed that the linear state and
  input inequality constraints, and the quadratic measure of
  performance, are all separable with respect to the spatial dimension
  of the underlying cascade of sub-systems, as well as the temporal
  dimension of the dynamics. By virtue of this structure, the
  computation cost of an interior-point method for an equivalent
  quadratic programming formulation of the optimal control problem can
  be made to scale linearly with the number of sub-systems. However,
  the complexity of this approach grows cubically with the time
  horizon. As such, computational advantage becomes apparent in
  situations where the number of sub-systems is relatively large. In
  any case, the method is amenable to distributed computation with low
  communication overhead and only immediate upstream neighbour sharing
  of partial model data among processing agents. An example is
  presented to illustrate an application of the main results to model
  data for the cascade dynamics of an automated irrigation channel.
\end{abstract}

\section{Introduction}
The application of model predictive control involves solving
finite-horizon optimal control problems in a receding horizon
fashion~\cite{garcia1989model,mayne2000constrained,
  maciejowski2002predictive,rawlings2009model}. When the penalty
function used to quantify performance and the inequality constraints
on the system states and inputs all separate along the prediction
horizon, additional structure in the equality constraints that encode
the system dynamics can be exploited to devise efficient methods for
computing the solutions. In particular, methods with computation costs
that scale linearly with the time horizon and cubically with the
number of states and inputs are
well-known~\cite{Wright93,Rao98,diehl2009efficient,WangBoyd10,domahidi2012efficient}.
The cubic scaling of these methods in the spatial dimension of the
problem, however, can be a limiting factor within the context of
controlling large-scale interconnections of sub-systems. In this
paper, interconnection structure is exploited over temporal structure
to devise a more scalable method for problems with cascade dynamics in
particular; i.e., when the system to control is the series
interconnection of numerous sub-systems, each with a control input, an
interconnection input, and an interconnection output. Such models arise
in the study of irrigation and drainage
networks~\cite{li2005water,Soltanian,puig2009predictive},
mutli-reservoir and hydro-power
systems~\cite{labadie2004optimal,Morari}, vehicle
platoons~\cite{Seiler2004}, and supply chain
management~\cite{disney2003effect}, for example.

The proposed method for solving finite-horizon optimal control
problems with cascade dynamics is closely related to the
interior-point method developed
in~\cite{Wright93,Rao98}. Interchanging the roles of the temporal and
spatial dimensions of such problems yields linear scaling of
computation cost with the number of sub-systems along the
cascade. However, the complexity grows cubically with the time
horizon, despite the causal flow of information in the temporal
dimension. The development illuminates the difficulty of overcoming
such cubic dependence. Computational advantage over methods that
exploit temporal structure, rather than the spatial structure
exploited here, arises when the length of the cascade is relatively
large compared to the prediction horizon. In any case, the method is
amenable to distributed computation over a linear network of
processing agents, one for each sub-system, with limited
neighbour-to-neighbour communication, and only partial sharing of
model information between neighbouring agents.

The paper is organized as follows. The class of finite-horizon optimal
control problems studied is defined in Section~\ref{sec:FHOC}. The
formulation of an equivalent quadratic program (QP) is given in
Section~\ref{sec:QP}. A scalable interior point method for computing
an optimal solution of the structured QP is developed in
Section~\ref{sec:IPM}. Proofs are deferred to the Appendix. Finally, a
numerical example based on model data for an automated irrigation
channel is presented Section~\ref{sec:examp}. Some concluding remarks
are provided in Section~\ref{sec:conc}.

\section{Problem Formulation}

A class of finite-horizon optimal control problems with cascade
dynamics is defined in this section. In addition to the directed
interconnection structure of the underlying cascade of sub-systems, a
defining characteristic of this class is the separability of the state
and input inequality constraints and performance index across both the
spatial and temporal dimensions of the system dynamics. An equivalent
QP with computationally favorable structure is formulated in
Section~\ref{sec:QP}.

\subsection{Constrained finite-horizon optimal control}
\label{sec:FHOC}


Consider the cascade of $N\in\mathbb{N}$ linear discrete-time
dynamical sub-systems modelled by
\begin{align}
  x_{j}(t+1) &= A_j(t) x_j(t) + B_j(t) u_j(t) + E_j(t) x_{j-1}(t), \label{eq:tempmodel} 
\end{align}
given initial conditions $x_j(0) = \xi_{j}\in\mathbb{R}^{n_j}$ and
model data $A_j(t)\in\mathbb{R}^{n_j\times n_j}$,
$B_j(t)\in\mathbb{R}^{n_j\times m_j}$, and
$E_j(t)\in\mathbb{R}^{n_j\times n_{j-1}}$, with $E_1(t)=0$ so that the
spatial boundary value $x_0(t)$ is effectively zero, for
$j=1,\ldots,N$ and $t=0,\ldots,T-1$. The parameter $T\in\mathbb{N}$ is
a specified time (or prediction) horizon. The problem of interest is
to set $u_j(t)$, for each sub-system index $j=1,\ldots,N$ and sample
time (indexed by $t=0,\ldots,T-1$), in order to minimize the separable
penalty function
\begin{align} \label{eq:cost}
  J = \frac{1}{2}\sum_{j=1}^N \left(  \left(\sum_{t=0}^{T-1} 
\begin{bmatrix} x_j(t)^\top & u_j(t)^\top  \end{bmatrix} 
\begin{bmatrix}
Q_j(t) & S_j(t)^\top \\
S_j(t) & R_j(t)
\end{bmatrix}
              \begin{bmatrix} x_j(t) \\ u_j(t)  \end{bmatrix} \right)
  + x_j(T)^\top P_j
  x_j(T) \right),
\end{align} 
subject to separable inequality constraints
\begin{align} \label{eq:constraints}
M_j(t) x_j(t) &+ L_j(t) u_j(t) \leq c_j(t)~\text{ for }~ j=1,\ldots,N,
                \text{ and } t=0,\ldots,T, 
\end{align}
where
\begin{align} \label{eq:QSRcond}
\begin{bmatrix}
Q_j(t) & S_j(t)^\top \\
S_j(t) & R_j(t)
\end{bmatrix}\succeq 0,
\end{align}
$M_j(t)\in\mathbb{R}^{\nu_j\times n_j}$,
$L_j(t)\in\mathbb{R}^{\nu_j\times m_j}$ and
$c_{j}(t)\in\mathbb{R}^{\nu_j}$, with
$0\prec R_j(t)=R_j(t)^\top\in\mathbb{R}^{m_j\times m_j}$, $M_j(0)=0$,
$L_j(T)=0$, $S_j(T)=0$, and $Q_j(T)=P_j=P_j^\top\succeq 0$, are given
for $t=0,\ldots,T$ and $j=1,\ldots, N$. Note that (\ref{eq:QSRcond})
implies $Q_j(t)=Q_j(t)^\top\succeq 0$ and
$Q_j(t) - S_j(t)^\top R_j(t)^{-1}S_j(t) \succeq 0$, since
$R_j(t)\succ 0$.

It is possible to reformulate the optimal control problem defined
above in a number of ways that result in standard QPs. Following the
style of QP reformulation in~\cite{Wright93} leads to an
interior-point method involving the solution of linear algebra
problems with favourable block tridiagonal structure. This is
exploited in Section~\ref{sec:IPM} to devise an algorithm with
per-iteration computation cost that scales linearly with cascade
length $N$.

\subsection{A QP formulation}
\label{sec:QP}
First note that the equality constraint corresponding to the dynamics
(\ref{eq:tempmodel}) can be reformulated as follows.
Define, for $j=1,\ldots,N$,
\begin{align*}
\hat{u}_j &= \begin{bmatrix} u_j(0)^\top & \cdots &
  u_j(T-1)^\top \end{bmatrix}^\top \in\mathbb{R}^{m_j T}
\quad\text{and}\quad
\hat{x}_j &= \begin{bmatrix} x_j(0)^\top & \cdots &
  x_j(T)^\top \end{bmatrix}^\top \in\mathbb{R}^{n_j (T+1)}. 
\end{align*}
Then 
\begin{align} \label{eq:spacemodel} -\hat{A}_j \hat{x}_j + \hat{E}_j
  \hat{x}_{j-1} + \hat{B}_j \hat{u}_j + \hat{H}_j \xi_{j} = 0,
\end{align}
where 
\begin{align*}
&\hat{A}_j = \begin{bmatrix}
I & 0 & \cdots & \cdots & 0 \\
-A_j(0) & I & \ddots &  & \vdots \\
0 & \ddots & \ddots & \ddots & \vdots\\
\vdots & \ddots & \ddots & \ddots & 0 \\
0 & \cdots & 0 & -A_j(T-1)& I
\end{bmatrix}, \quad
\hat{E}_j = \begin{bmatrix}
0 & \cdots  & \cdots & \cdots & 0 \\
E_j(0) & \ddots  &   & & \vdots \\
0 & E_j(1) & \ddots  &   & \vdots \\
\vdots &  \ddots & \ddots & \ddots &  \vdots \\
0 & \cdots & 0 & E_j(T-1) & 0  
\end{bmatrix},\\
&\hat{B}_j = \begin{bmatrix}
0 & \cdots & \cdots & 0 \\
B_j(0) & \ddots    &    &   \vdots \\
0 & \ddots & \ddots &   \vdots \\
\vdots &  \ddots   & \ddots &  0 \\
0 & \cdots & 0 & B_j(T-1)   
\end{bmatrix}, \quad\text{and}\quad
\hat{H}_j = \begin{bmatrix} I \\ 0 \\ \vdots \\ \vdots \\ 0 \end{bmatrix},
\end{align*}
for $j=1,\ldots,N$. The boundary condition for (\ref{eq:spacemodel})
is effectively $\hat{x}_{0}=0$ as $\hat{E}_1=0$. Moreover, with
\begin{align*} 
 \hat{Q}_j &= \diag(Q_j(0),\ldots,Q_j(T)) \in \mathbb{R}^{n_j(T+1)
  \times n_j(T+1)},\\ 
  \hat{R}_j &= \diag(R_j(0),\ldots,R_j(T-1)) \in \mathbb{R}^{m_jT \times
              m_jT}, \\ 
 \hat{S}_j &= \begin{bmatrix} 
              \diag(S_j(0),\ldots, S_j(T)) & 0 \end{bmatrix}
                          \in \mathbb{R}^{m_jT\times n_j (T+1)},\\
  \hat{M}_j &=\diag(M_j(0),\ldots,M_j(T)) 
  \in \mathbb{R}^{\nu_j(T+1) \times n_j(T+1)},\\
  \hat{L}_j &= \begin{bmatrix} 
    \diag(L_j(0),\ldots, L_j(T-1))^\top & 0 \end{bmatrix}^\top \in \mathbb{R}^{\nu_j (T+1) \times
  m_jT},
\end{align*}
and
$\hat{c}_j = \begin{bmatrix} c_j^\top(0) & \cdots &
  c_j^\top(T) \end{bmatrix}^\top\in \mathbb{R}^{\nu_j
  (T+1)}$,
the performance index and constraints can be reformulated as
\begin{align} \label{eq:cost2} J = \frac{1}{2}\sum_{j=1}^N
  \left(\hat{x}_j^\top \hat{Q}_j \hat{x}_j + \hat{u}_j^\top \hat{R}_j
    \hat{u}_j+2\hat{u}_j^\top \hat{S}_j \hat{x}_j\right)
\end{align}
and
\begin{align} \label{eq:constraints2} \hat{M}_j \hat{x}_j + \hat{L}_j
  \hat{u}_j - \hat{c}_j \leq 0 \text{ for } j=1,\ldots,N,
\end{align}
respectively.  To summarize, the problem of minimizing (\ref{eq:cost})
subject to the equality constraints (\ref{eq:tempmodel}), and
inequality constraints (\ref{eq:constraints}), is now in the form of
the following standard QP:
\begin{align*}
\min_{\substack{ {(\hat{u}_1,\ldots,\hat{u}_N)\in\mathbb{R}^{m_1
  T}\times\cdots\times \mathbb{R}^{m_NT}} \\ {
(\hat{x}_1,\ldots,\hat{x}_N)\in\mathbb{R}^{n_1
  (T+1) }\times\cdots\times \mathbb{R}^{n_N (T+1) }}}}
  (\ref{eq:cost2})
\quad  \text{ subject to }\quad (\ref{eq:spacemodel}) \text{ and } (\ref{eq:constraints2}). 
\end{align*}

\section{Developing an interior-point method that scales linearly
  with cascade length} \label{sec:IPM}

A primal-dual interior-point method for solving a QP involves the
application of Newton's iterative method to solve the
Karush--Kuhn--Tucker (KKT) conditions~\cite{Mehrotra92,Wright97}.
For the case at hand, the KKT conditions take the following form, for
$j=1,\ldots,N$, where $\hat{x}_0=0$ (since $\hat{E}_1=0$),
$\hat{E}_{N+1}=0$,
$\Lambda_j=\diag((\lambda_{j})_1,\ldots,(\lambda_{j})_{\nu_j})$,
$\Theta_j=\diag((\theta_{j})_{1},\ldots,(\theta_{j})_{\nu_j})$, and
$\mathbf{1}$ denotes a column vector of ones:
\begin{align*}
  \hat{Q}_j\hat{x}_j + \hat{S}_j^\top\hat{u}_j - \hat{A}_j^\top p_j 
  + \hat{M}_j^\top \lambda_j + \hat{E}_{j+1}^\top p_{j+1} 
  &=  0; \\
  \hat{S}_j \hat{x}_j+\hat{R}_j\hat{u}_j + \hat{B}_j^\top p_j + \hat{L}_j^\top \lambda_j 
  &=  0;\\
  -\hat{A}_j\hat{x}_j + \hat{E}_j\hat{x}_{j-1} + \hat{B}_j \hat{u}_j + 
  \hat{H}_j\xi_j
  &=  0;\\
  \hat{M}_j \hat{x}_j + \hat{L}_j \hat{u}_j - \hat{c}_j + \theta_j
  &=  0;\\
  \Lambda_j \Theta_j \mathbf{1}
  &=  0;\\
\begin{bmatrix} \lambda_i^\top & \theta_j^\top \end{bmatrix}^\top &\geq 0.
\end{align*}
Given
$s[k] = \begin{bmatrix} s_1[k]^\top & \cdots &
  s_N[k]^\top \end{bmatrix}^\top$, with
$s_j[k]= \begin{bmatrix} \hat{x}_j[k]^\top & \hat{u}_j[k]^\top &
  p_j[k]^\top & \lambda_j[k]^\top & \theta_j[k]^\top
\end{bmatrix}^\top$
for $j=1,\ldots,N$, the iteration at Newton step $k\in\mathbb{N}$ is
given by
\begin{align} \label{eq:iteration}
  s[k+1] =  s[k] + \alpha[k] \cdot \delta[k],
\end{align}
where $\alpha[k] > 0$ is a sufficiently small step-size parameter
selected online so that $\lambda_j[k+1] \geq 0$, $\theta_j[k+1] > 0$,
and
$\delta[k]=\begin{bmatrix} \delta_1[k]^\top & \cdots &
  \delta_N[k]^\top \end{bmatrix}^\top$
is the solution of the linearized KKT conditions
\begin{align} \label{eq:linearized}
\! \begin{bmatrix}
D_1[k] & -\Upsilon_2^\top & 0 & \ldots & 0 \\
-\Upsilon_2 & D_2[k] & -\Upsilon_3^\top & \ddots & \vdots \\
0 & -\Upsilon_3 & \ddots & \ddots & 0 \\
\vdots & \ddots    & \ddots &  D_{N-1}[k]  & -\Upsilon_N^\top \\
0 & \cdots & 0 & -\Upsilon_N & D_N[k] 
\end{bmatrix}
\!\! \begin{bmatrix}
\delta_1[k] \\ \delta_2[k] \\ 
\vdots \\  
\vdots \\
\delta_N[k]
\end{bmatrix} = \begin{bmatrix}
\rho_1[k] \\ \rho_2[k] \\ 
\vdots  \\ 
\vdots \\
\rho_N[k]
\end{bmatrix}\!,
\end{align}
and the following hold for $j=1,\ldots,N$, with $\Upsilon_{N+1}=0$
(since $\hat{E}_{N+1}=0$):
\begin{align}
  D_j[k] &=
\begin{bmatrix}
\hat{Q}_j & \hat{S}_j^\top & -\hat{A}_j^\top & \hat{M}_j^\top & 0 \\
\hat{S}_j & \hat{R}_j & \hat{B}_j^\top & \hat{L}_j^\top & 0 \\
-\hat{A}_j & \hat{B}_j & 0 & 0 & 0 \\
\hat{M}_j & \hat{L}_j & 0 & 0 & I \\
0 & 0 & 0 & \Theta_j[k] & \Lambda_j[k]
\end{bmatrix}; \label{eq:Dj}
\end{align}
\begin{align}
\Upsilon_j &=
\begin{bmatrix}
0 & 0 & 0 & 0 & 0 \\
0 & 0 & 0 & 0 & 0 \\
-\hat{E}_j & 0 & 0 & 0 & 0 \\
0 & 0 & 0 & 0 & 0 \\
0 & 0 & 0 & 0 & 0 
\end{bmatrix}; \label{eq:Upsilonj}
\end{align}
\begin{align}
\rho_j[k] =-\Upsilon_j s_{j-1}[k] + D_j[k] s_j[k] - \Upsilon_{j+1}^\top
s_{j+1}[k]
+ \begin{bmatrix} 0^\top & 0^\top & (\hat{H}_j\xi_j)^\top &
\hat{c}_j^\top & \sigma_j[k]^\top\end{bmatrix}^\top;
\end{align}
$\sigma_j[k] = -\Theta_j[k]\lambda_j[k] - \Lambda_j[k] \theta_j[k] +
\Lambda_j[k] \Theta_j[k] \mathbf{1} - \bar{\sigma} \mu[k] \mathbf{1}$;
$\Lambda_j[k] =
\diag((\lambda_j[k])_1,\ldots,(\lambda_{j}[k])_{\nu_j})$;
$\Theta_j[k] = \diag((\theta_{j}[k])_1,\ldots,(\theta_{j}[k])_{\nu_j})$;
$\bar{\sigma} \in (0,1)$ is a centring parameter; and
$\mu[k] = (\sum_{j=1}^N (\lambda_j[k])^\top \theta_j[k])/(\sum_{j=1}^N
\nu_j)$
is the duality gap. Given an appropriate initialization $s[1]$, the
linear equation (\ref{eq:linearized}) has a unique solution for each
iteration (\ref{eq:iteration}), as seen below. Construction of this
solution is facilitated by the block tridiagonal structure; e.g.,
see~\cite{Meurant92} 
and~\cite{Bevilacqua88}. Indeed, the computation cost of solving the
set of equations (\ref{eq:linearized}) at each iteration can be made
to scale linearly with $N$; ignoring structure would incur order $N^3$
complexity. Proofs of the following results, which underpin this
assertion, can be found in the Appendix.

\begin{lemma} \label{lem:solved} Dropping explicit dependence on
  the algorithm iteration index $k\in\mathbb{N}$, the unique solution
  of~\eqref{eq:linearized} is given by the backward and forward
  recursions
\begin{align*}
\tilde{\rho}_{j}
\!=\!
\begin{cases}
\rho_{N} & \text{for } j\!=\!N  \\
\rho_{j}+\Upsilon_{j+1}^\top \Sigma_{j+1}^{-1}\tilde{\rho}_{j+1} & \text{for } j\!=\!N-1,\dots,1
\end{cases}\!,
~\text{and }~
\delta_j\!=\!
\begin{cases}
\Sigma_1^{-1}\tilde{\rho}_1& \text{for } j\!=\!1 \\
\Sigma_j^{-1}(\tilde{\rho}_j+\Upsilon_j\delta_{j-1})  & \text{for } j\!=\!2,\dots,N
\end{cases}\!,
\end{align*}
respectively, where 
\begin{align}
\label{eq:sigma}
\Sigma_j &= \begin{cases}
D_N & \text{for } j=N\\
D_j - \Upsilon_{j+1}^\top(\Sigma_{j+1})^{-1} \Upsilon_{j+1} &
\text{for } j=N-1,\ldots,1
\end{cases}.
\end{align}
In particular, each $\Sigma_{j}\in\mathbb{R}^{p_j\times p_j}$, where
$p_j=((2n_j+2\nu_j)(T+1)+m_jT)$, is non-singular as required.
\end{lemma}

\begin{remark} Using Lemma~\ref{lem:solved} to solve
  (\ref{eq:linearized}) incurs a computation cost that scales linearly
  with the number of sub-systems $N$. Unfortunately, the $11$-block of
  $\Sigma_j$ (c.f. $\check{Q}_j\succeq 0$ in the proof of
  Lemma~\ref{lem:solved}) becomes a full matrix for $j=N-1,\ldots,1$,
  despite the sparsity of $\hat{Q}_j$, $\hat{R}_j$, $\hat{S}_j$,
  $\hat{A}_j$, $\hat{B}_j$, $\hat{M}_j$, $\hat{L}_j$, $\Theta_j$, and
  $\Lambda_j$. Therefore, inverting $\Sigma_j$ in the recursions above
  incurs cost that scales cubically with $T$, yielding an overall
  computation cost of order $NT^3$.
\end{remark}
%

The following result encapsulates a method for inverting the matrix
$\Sigma_j$. This method involves the inverses of block diagonal and
block bi-diagonal matrices of order $T\times T$, and the inverse of
one unstructured positive-definite $(m_jT) \times (m_jT)$ matrix; note
that typically $m_j$ is smaller than $n_j$ and $\nu_j$. The
computation cost is order $T$ and $T^3$, respectively, for an overall
cost that scales cubically with the time horizon.
The matrix $\hat{Q}_j$ is not required to be non-singular, as needed
to follow steps used to invert similarly structured matrices
in~\cite{WangBoyd10,domahidi2012efficient}, for example. Furthermore,
the method is amenable to distributed implementation with low
communication overhead, as also discussed in more detail subsequently.
\begin{lemma} 
  \label{lem:invertSigma} With $\Sigma_j$ defined according to
  \eqref{eq:sigma}, the unique solution of the linear equations
\begin{align} \label{eq:tosolve}
\Sigma_j
\begin{bmatrix}
X_1 \\ X_2 \\ X_3 \\ X_4 \\ X_5
\end{bmatrix}
= 
\begin{bmatrix}
\hat{Q}_j-\hat{E}_{j+1}^\top (\Sigma_{j+1}^{-1})_{33}\hat{E}_{j+1} &
\hat{S}_j^\top &  -\hat{A}_j^\top & \hat{M}_j^\top & 0 \\
\hat{S}_j & \hat{R}_j & \hat{B}_j^\top & \hat{L}_j^\top & 0 \\
-\hat{A}_j & \hat{B}_j & 0 & 0 & 0 \\
\hat{M}_j & \hat{L}_j & 0 & 0 & I \\
0 & 0 & 0 & \Theta_j & \Lambda_j
\end{bmatrix}
\begin{bmatrix}
X_1 \\ X_2 \\ X_3 \\ X_4 \\ X_5
\end{bmatrix}
=
\begin{bmatrix}
Y_1 \\ Y_2 \\ Y_3 \\ Y_4 \\ Y_5
\end{bmatrix},
\end{align}
given $Y_1\in\mathbb{R}^{n_j(T+1)\times q}$,
$Y_2\in\mathbb{R}^{m_jT \times q}$,
$Y_3\in\mathbb{R}^{\nu_j(T+1) \times q}$,
$Y_4\in\mathbb{R}^{\nu_j(T+1) \times q}$,
$Y_5\in\mathbb{R}^{\nu_j(T+1) \times q}$, can be constructed as
follows, for $j=1,\ldots,N$ with $\hat{E}_{N+1}=0$:
\begin{align*}
\begin{bmatrix}
X_1 \\ X_2 \\ X_3
\end{bmatrix}
&= \Psi_j \Omega_j \Psi_j^\top 
\left(
\begin{bmatrix}
Y_1\\ Y_2 \\ Y_3
\end{bmatrix}
+\begin{bmatrix}
\hat{M}_j^\top \Theta_j^{-1}\Lambda_j & -\hat{M}_j^\top \Theta_j^{-1} \\ 
\hat{L}_j^\top \Theta_j^{-1}\Lambda_j & -\hat{L}_j^\top \Theta_j^{-1} \\ 
0 & 0
\end{bmatrix}
\begin{bmatrix}
Y_4 \\ Y_5
\end{bmatrix}
\right) \quad \text{ and } \\
\begin{bmatrix}
X_4 \\ X_5 
\end{bmatrix}
&= 
\begin{bmatrix} 
-\Theta_j^{-1}\Lambda_j & \Theta_j^{-1} \\ I & 0
\end{bmatrix} \left(
-\begin{bmatrix} 
\hat{M}_j & \hat{L}_j \\
0 & 0 
\end{bmatrix}
\begin{bmatrix}
X_1 \\ X_2
\end{bmatrix}
+ 
\begin{bmatrix}
Y_4 \\ Y_5
\end{bmatrix}
\right) 
\end{align*}
where
\begin{align}
\label{eq:PsiInv}
\Psi_j &= 
\begin{bmatrix}
I & 0 & 0 \\ 
-\tilde{R}_j^{-1}\tilde{S}_j & I & -\tilde{R}_j^{-1}\tilde{B}_j^\top\tilde{A}_j^{-\top} \\
0 & 0 & \tilde{A}_j^{-\top}
\end{bmatrix},
\\
\label{eq:OmegaInv}
\Omega_j  &=
\begin{bmatrix}
\tilde{Z}_j
  & 0 &
-\tilde{W}_j^\top
\\
  0 & \tilde{R}_j^{-1} & 0 \\
  -\tilde{W}_j & 0 & -\tilde{W}_j\tilde{Q}_j
\end{bmatrix},
\\
\label{eq:Ztil}
\tilde{Z}_j &=
  \tilde{A}^{-1}\tilde{B}
  (\tilde{R}
   +\tilde{B}^\top\tilde{A}^{-\top}\tilde{Q}\tilde{A}^{-1}\tilde{B})^{-1} 
  \tilde{B}^\top\tilde{A}^{-\top},
\\
\label{eq:Wtil}
\tilde{W}_j &= I - \tilde{Q}_j\tilde{Z}_j,\\
\label{eq:Htil}
\tilde{H}_j &=
\begin{bmatrix}
\bar{Q}_j & \tilde{S}_j^\top \\ \tilde{S}_j & \tilde{R}_j 
\end{bmatrix}
= 
\begin{bmatrix}
\hat{Q}_j
-
\hat{E}_{j+1}^\top(\Sigma_{j+1}^{-1})_{33}\hat{E}_{j+1} & \hat{S}_j^\top \\ 
\hat{S}_j & \hat{R}_j
\end{bmatrix} + 
\begin{bmatrix}
\hat{M}_j^\top \\ \hat{L}_j^\top
\end{bmatrix}\Theta_j^{-1}\Lambda_j 
\begin{bmatrix} \hat{M}_j & \hat{L}_j \end{bmatrix} \succeq 0,
\end{align}
$\tilde{Q}_j = \bar{Q}_j-\tilde{S}_j^\top\tilde{R}_j^{-1}\tilde{S}_j \succeq
0$, $\tilde{A}_j=\hat{A}_j+\hat{B}_j\tilde{R}_j^{-1}\tilde{S}_j$ is
non-singular and $\tilde{B}_j=\hat{B}_j$.
\end{lemma}
\begin{remark}
  The cost of computing $\tilde{A}_j^{-1}$ is order $T$ because of the
  (lower) block bi-diagonal structure of $\tilde{A}_j$. The
  $(m_jT)\times (m_jT)$ symmetric positive semi-definite matrix
  $(\tilde{R}_j + \tilde{B}_j^\top \tilde{A}_j^{-\top} \tilde{Q}_j
  \tilde{A}_j^{-1}\tilde{B}_j)\succ 0$
  is full, on the other hand. So the computation cost of inversion is
  order $T^3$, whereby an approach to computing $\Sigma_j$ based on
  Lemma~\ref{lem:invertSigma} scales as $T^3$. An alternative approach
  based on inversion of $\tilde{H}_j$, assuming that
  it is non-singular by requiring $$\left[ \begin{smallmatrix} \hat{Q}_j & \hat{S}_j^\top \\
      \hat{S}_j & \hat{R}_j \end{smallmatrix} \right]\succ 0,$$
  for example, would also involve computation cost that scales as
  $T^3$ since the $(n_j(T+1))\times(n_j(T+1))$ matrix
  $\bar{Q}_j= \hat{Q}_j -
  \hat{E}_{j+1}^\top(\Sigma_{j+1}^{-1})_{33}\hat{E}_{j+1}$
  in the $11$-block of $\tilde{H}_j$ is full for $j<N$. By constrast,
  as seen in~\cite{WangBoyd10,domahidi2012efficient}, for example,
  taking such an approach can be fruitful within the context of
  exploiting temporal structure in optimal control problems, since the
  $11$-block (i.e., $Q$ block) of the corresponding $\Sigma$ matrix is
  sparse in this case.
\end{remark}

Finally, it is of note that the calculations required to implement the
solution of (\ref{eq:linearized}) according to Lemma~\ref{lem:solved}
can be distributed among processing agents, one for each sub-system,
connected in a linear network that mirrors the underlying cascade of
sub-system models. Indeed, the agent associated with sub-system
$j\in\{2,\ldots,N-1\}$ needs to send $\Upsilon_j\Sigma_j^{-1}\tilde{\rho}_j$
to, and receive $\Upsilon_{j}\delta_{j-1}$ from, the agent for $j-1$.
Moreover, the processing agent for subsystem $j\in\{1,\ldots,N-1\}$
only needs access to $\hat{E}_{j+1}$ (i.e., the influence of $j$ on
$j+1$) and no further model, performance index or constraint data from
other sub-systems. The processing agent for $j=1$ need only
communicate with the agent for $j=2$ and the processing agent for
$j=N$ need only communicate with the agent for $j=N-1$. For each
Newton iteration of the interior-point method, information is
communicated in a sequence of steps, once up the cascade, and then
once down the cascade, without further iteration. Calculation of the
the duality gap and step size updates can be determined by a similar
pattern of up and then down exchange of information. Agent pipelining
by the provision of buffering in the communication could be exploited
to improve throughput, without of course improving latency, which may
degrade slightly. Of course, it is necessary to process a sufficient
number of Newton updates in the interior-point method. One of the
appealing features of interior point methods is the typically small
number of iterations (e.g., ten to fifteen) needed to reach a good
solution to the QP. As such, the overall communication overhead of a
distributed version of the method is low. The localization of model
and other problem data can also be considered advantageous from
security and privacy perspectives.

\section{Example} \label{sec:examp} Numerical results are obtained by
applying the preceding developments to model data for an automated
irrigation channel, within the context of a water-level reference
planning problem.
A state-space model of the form (\ref{eq:tempmodel}) can be
constructed for channels that operate under decentralized
distant-downstream control~\cite{Cantoni07,Soltanian}; each sub-system
corresponds to a stretch of channel between adjacent flow regulators
called a pool.
An optimal control problem can be formulated to determine the
water-level reference input for each pool, across a planning horizon
for which a load forecast may be known, subject to hard constraints on
the water-levels and flows.
For this example, the model data of the pools, the distributed
controllers, the discretization sample-period, and constraint levels
are all borrowed from~\cite{PaperAmir}; the many pool channels
considered here are constructed by concatenating sections of the
channel considered there. In the model for each sub-system (i.e.,
pool) there is one control input, the water-level reference, and four
states, including two for the water-level dynamics and two for the
PI-type feedback controller that sets the upstream inflow on the basis
of the measured downstream water-level error. Box type constraints on
the water-level and controller flow output states are to be satisfied
in addition to the water-level reference input constraints. For each
sub-system $j\in\{1,\ldots,N\}$, the following hold: $n(j)=4$;
$m(j)=1$; and $\nu(j)=6$. The other model parameters (e.g., entries of
state-space matrices) are not uniform along the channel.

\begin{figure}[t]
\centering
\begin{tikzpicture}
\node[] at (0,0) {\includegraphics[width=\linewidth]{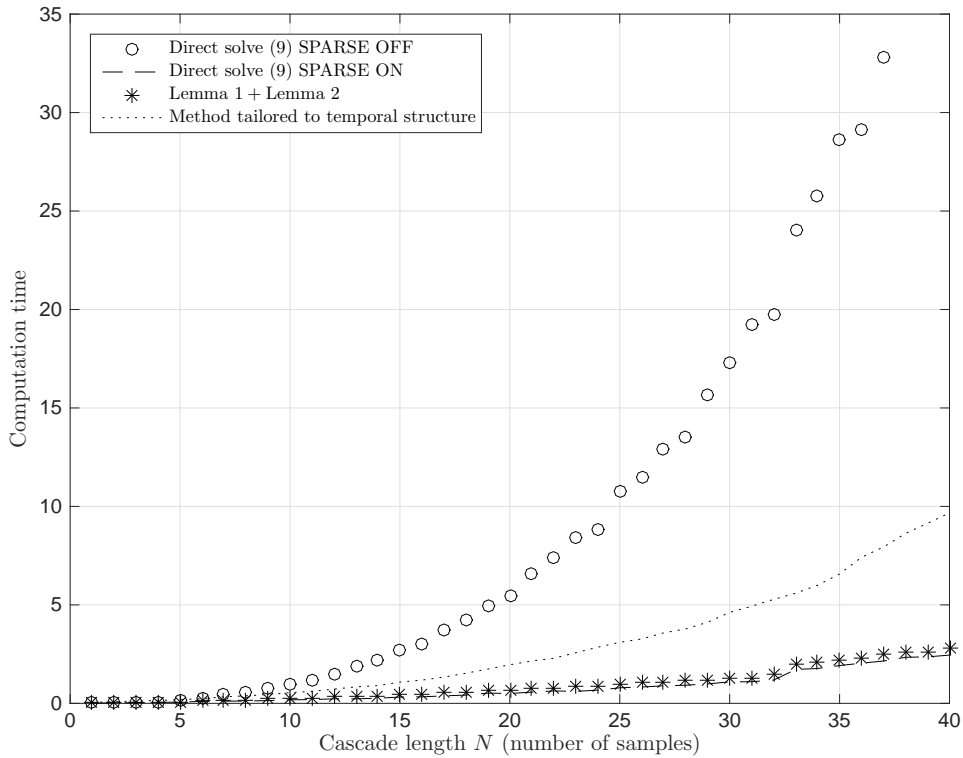}};
\end{tikzpicture}
\caption{\label{fig:spatialvarn} Computation time for 16 Newton
  iterations of the interior point method in the following cases for
  varying $N$ and fixed $T=5$:
  $(i)$ Solving (\ref{eq:linearized}) directly with sparsity exploitation
  disabled (circle); $(ii)$ Solving (\ref{eq:linearized}) directly with
  sparsity exploitation enabled (dashed); $(iii)$ Solving (\ref{eq:linearized})
  via Lemma~\ref{lem:solved} and Lemma~\ref{lem:invertSigma} (star); and
  $(iv)$ Solving (\ref{eq:linearized}) with a method tailored to the
  temporal structure (dot).}
\end{figure}

Figure~\ref{fig:spatialvarn} shows the MATLAB~2014b computation time
(in seconds with forced single computation thread) of exactly sixteen
iterations of the interior point method described above for different
ways of solving (\ref{eq:linearized}). The number of sub-systems $N$
is varied from $1$ to $40$, with fixed time-horizon $T=5$. In all
cases the duality gap is less than $10^{-3}$ after sixteen
iterations. The following cases are considered: $(i)$ Direct solution
of (\ref{eq:linearized}) with sparsity exploitation disabled by
perturbing the block tridiagonal matrix $X$ on the left-hand side
(i.e., \texttt{(X+eps)$\backslash$rho}); $(ii)$ Direct solution of
(\ref{eq:linearized}) with sparsity exploitation enabled (i.e.,
\texttt{X=sparse(X);~X$\backslash$rho}); $(iii)$ Solution of
(\ref{eq:linearized}) via Lemma~\ref{lem:solved} and
Lemma~\ref{lem:invertSigma}; and $(iv)$ Solution of
(\ref{eq:linearized}) via a method tailored to exploit the temporal
structure also present in $X$, along the lines of
Lemma~\ref{lem:solved}. As expected, the approach that does not
exploit structure incurs a computation time that grows as $N^3$. The
use of Lemma~\ref{lem:solved} with Lemma~\ref{lem:invertSigma}, on the
other hand, scales linearly with $N$. Moreover, the performance
achieved is as good as enabling MATLAB to exploit structure when
solving (\ref{eq:linearized}) directly, which of course requires the
MATLAB environment. Also note that a method tailored to the temporal
structure of $X$ does not scale linearly.
By contrast, Figure~\ref{fig:spatialvarT} shows the computation time
for fixed $N=10$ and varying time-horizon $T$. As expected, the
approaches in cases $(i)$ and $(iii)$ scale as $T^3$.
This is similar to the cubic scaling of computation cost with the
dimension of the state in methods that are tailored to exploit the
temporal structure of optimal control problems;
e.g.,~\cite{WangBoyd10,domahidi2012efficient}. Direct solution of
(\ref{eq:linearized}) with sparsity exploitation enabled appears to
asymptotically scale linearly with $T$, in a way that is consistent
with the method that is tailored to exploit temporal structure.

\begin{figure}[tbp]
\centering
\begin{tikzpicture}
  \node[] at (0,0) {\includegraphics[width=\linewidth]{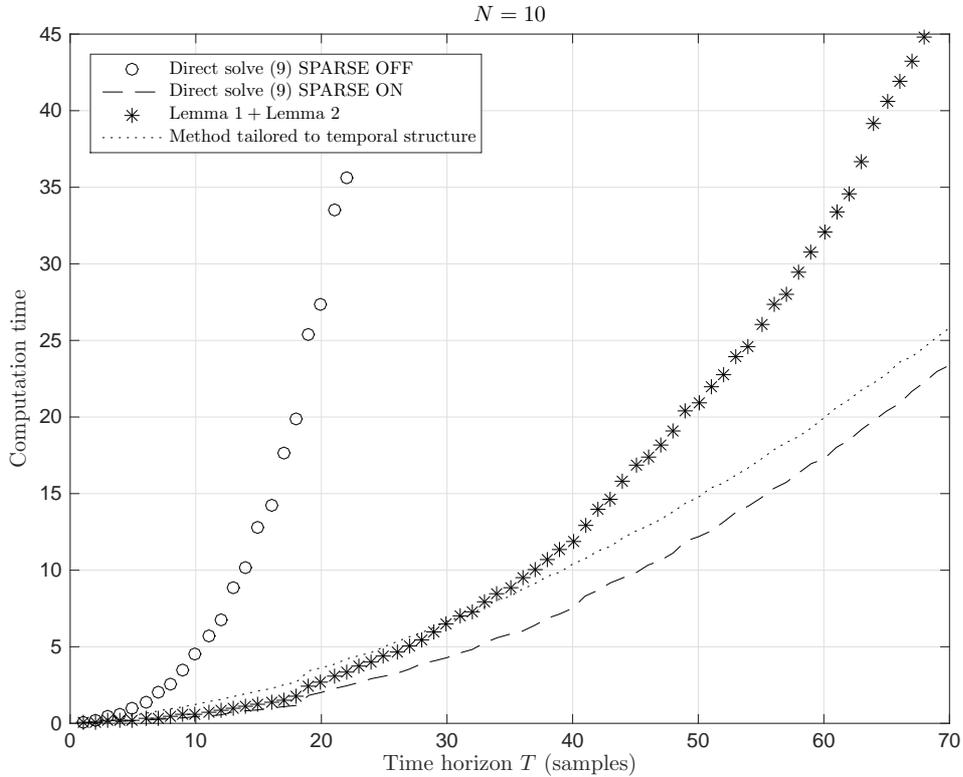}};
\end{tikzpicture}
\caption{\label{fig:spatialvarT} Computation time for 16 Newton
  iterations of the interior point method in the following cases for
  varying $T$ and fixed $N=10$: $(i)$ Solving (\ref{eq:linearized})
  directly with sparsity exploitation disabled (circle); $(ii)$ Solving
  (\ref{eq:linearized}) directly with sparsity exploitation enabled
  (dashed); $(iii)$ Solving (\ref{eq:linearized}) via
  Lemma~\ref{lem:solved} and Lemma~\ref{lem:invertSigma} (star); and
  $(iv)$ Solving (\ref{eq:linearized}) with a method tailored to the
  temporal structure (dot).}
\end{figure}

\section{Conclusion}
\label{sec:conc}

The main contribution is a scalable interior-point method for
computing the solution of constrained discrete-time optimal control
problems with cascade dynamics, over a finite prediction horizon. By
exploiting the spatial structure arising from the cascade dynamics,
the computation cost of each step scales linearly with the number of
sub-systems along the cascade. By constrast, the method exhibits cubic
growth of computation time as the prediction horizon increases. Direct
application of standard methods, typically tailored to exploit the
temporal structure of optimal control problems in order to achieve
linear scaling with the time horizon, would yield complexity that
grows as the cube of the number of system states, and thus, the number
of sub-systems. The main developments are illustrated by numerical
example on model data for an automated irrigation channel. A topic of
ongoing research pertains to extending the main ideas to exploit
directed and undirected spatial propagation of information in tree
networks of dynamical systems. Another concerns the design of custom
hardware for distributed algorithm implementation.

\appendix

\section{A technical lemma and proofs}

\begin{lemma} \label{lem:QRS} Given
  $Q=Q^\top\in\mathbb{R}^{n\times n}$,
  $R=R^\top\in\mathbb{R}^{m\times m}$, $S\in\mathbb{R}^{m\times n}$,
  $A\in\mathbb{R}^{n\times n}$, $B\in\mathbb{R}^{n\times m}$,
  $M\in\mathbb{R}^{\nu \times n}$, $L\in\mathbb{R}^{\nu \times m}$, suppose
  that $R\succ 0$, $Q\succeq 0$, $Q-S^\top R^{-1} S\succeq 0$,
  and $\tilde{A}=A+BR^{-1}S$ is non-singular. Let
\begin{align*}
D = \begin{bmatrix}
Q & S^\top & -A^\top & M^\top & 0 \\
S & R & B^\top & L^\top & 0 \\
-A & B & 0 & 0 & 0 \\
M & L & 0 & 0 & I \\
0 & 0 & 0 & \Theta & \Lambda
\end{bmatrix},
\end{align*}
where $\Theta=\diag(\theta_1,\ldots,\theta_\nu)\succ 0$ and
$\Lambda=\diag(\lambda_1,\ldots,\lambda_\nu)\succeq 0$. Then $D$ is
non-singular. Moreover, the $33$-block of $D^{-1}$ is given by 
\begin{align*}
-\tilde{A}^{-\top}\tilde{Q}^{\frac{1}{2}}(I +
  \tilde{Q}^{\frac{1}{2}}\tilde{A}^{-1} \tilde{B} \tilde{R}^{-1}
  \tilde{B}^\top \tilde{A}^{-\top}
  \tilde{Q}^{\frac{1}{2}})^{-1}\tilde{Q}^{\frac{1}{2}}\tilde{A}^{-1}\preceq 0,
\end{align*}
where 
\begin{align*}
\begin{bmatrix}
\bar{Q} & \tilde{S}^\top \\ \tilde{S} & \tilde{R} 
\end{bmatrix}
= 
\begin{bmatrix}
Q & S^\top \\ S & R
\end{bmatrix} + 
\begin{bmatrix}
M^\top \\ L^\top
\end{bmatrix}\Theta^{-1}\Lambda 
\begin{bmatrix} M & L \end{bmatrix} \succeq 0,
\end{align*}
$\tilde{Q} = \bar{Q}-\tilde{S}^\top\tilde{R}^{-1}\tilde{S} \succeq 0$ and
$\tilde{B}=B$.
\end{lemma}
\begin{proof}
  First note that $\Theta\succ 0$ is invertible and that
  $$\begin{bmatrix} 0 & I \\ \Theta & \Lambda \end{bmatrix}^{-1}
  = \begin{bmatrix} -\Theta^{-1}\Lambda & \Theta^{-1} \\ I &
    0 \end{bmatrix}.$$
  As such, it follows that $D$ is invertible
  if and only if the Schur complement
\begin{align*}
\begin{bmatrix} 
\bar{Q} & \tilde{S}^\top & -A^\top \\ \tilde{S} & \tilde{R} & B^\top \\ -A & B & 0
\end{bmatrix} 
= 
\begin{bmatrix} 
Q & S^\top & -A^\top \\ S & R & B^\top \\ -A & B & 0
\end{bmatrix}
- \begin{bmatrix}
M^\top & 0 \\ L^\top & 0 \\ 0 & 0
\end{bmatrix}
\begin{bmatrix} 0 & I \\ \Theta & \Lambda \end{bmatrix}^{-1}
\begin{bmatrix} M & L & 0 \\ 0 & 0 & 0 \end{bmatrix}
\end{align*}
is non-singular. Moreover, the inverse of this matrix, when it exists,
is precisely the matrix comprising the first three block rows and
columns of $D^{-1}$; as such, the $33$-blocks coincide.

Now note that $\tilde{R} = R + L^\top\Theta^{-1}\Lambda L \succ 0$,
since the diagonal matrix $\Theta^{-1}\Lambda \succeq 0$ and
$R \succ 0$, and
\begin{align} \label{eq:QQinv}
&\begin{bmatrix}
\tilde{Q} & 0 & -I \\
0 & \tilde{R} & 0 \\
-I & 0 & -\tilde{A}^{-1}\tilde{B}\tilde{R}^{-1}\tilde{B}^\top
\tilde{A}^{-\top}
\end{bmatrix}
\\
&\quad\quad\quad\quad 
=
\begin{bmatrix}
I & -\tilde{S}^\top\tilde{R}^{-1} & 0 \\
0 & I & 0 \\
0 & -\tilde{A}^{-1}\tilde{B}\tilde{R}^{-1} & \tilde{A}^{-1}
\end{bmatrix}
\begin{bmatrix} 
\bar{Q} & \tilde{S}^\top & -A^\top \\ \tilde{S} & \tilde{R} & B^\top \\ -A & B & 0
\end{bmatrix} 
\begin{bmatrix}
I & 0 & 0 \\ 
-\tilde{R}^{-1}\tilde{S} & I & -\tilde{R}^{-1}\tilde{B}^\top\tilde{A}^{-\top} \\
0 & 0 & \tilde{A}^{-\top}
\end{bmatrix}, \nonumber
\end{align}
where
$\tilde{Q} = \bar{Q} - \tilde{S}^\top \tilde{R}^{-1}\tilde{S}\succeq
0$ because
\begin{align*}
\begin{bmatrix}
\bar{Q} & \tilde{S}^\top \\
\tilde{S} & \tilde{R}
\end{bmatrix} = 
\begin{bmatrix}
Q & S^\top \\ S & R
\end{bmatrix} + \begin{bmatrix} M^\top \\ L^\top \end{bmatrix}
\Theta^{-1}\Lambda \begin{bmatrix} M \\ L \end{bmatrix}\succeq 0.
\end{align*}
Therefore, $D$ is non-singular if and only if 
\begin{align} \label{eq:Qinv}
\begin{bmatrix}
\tilde{Q} & -I \\ -I & -\tilde{A}^{-1}\tilde{B}\tilde{R}^{-1}\tilde{B}^\top \tilde{A}^{-\top}
\end{bmatrix}
= 
\begin{bmatrix}
I & -\tilde{Q} \\ 0 & I
\end{bmatrix}
\begin{bmatrix}
0 & -(I+\tilde{Q}\tilde{A}^{-1}\tilde{B}\tilde{R}^{-1}\tilde{B}^\top \tilde{A}^{-\top})\\
-I & -\tilde{A}^{-1}\tilde{B}\tilde{R}^{-1}\tilde{B}^\top \tilde{A}^{-\top}
\end{bmatrix}
\end{align}
is non-singular, which is the case if and only if
$(I+\tilde{Q}\tilde{A}^{-1}\tilde{B}\tilde{R}^{-1}\tilde{B}^\top \tilde{A}^{-\top})$ is
non-singular, or equivalently, if and only if
$-1\notin\mathrm{spec}(\tilde{Q}\tilde{A}^{-1}\tilde{B}\tilde{R}^{-1}\tilde{B}^\top
\tilde{A}^{-\top})$. That later
holds because
$$\mathrm{spec}(\tilde{Q}\tilde{A}^{-1}\tilde{B}\tilde{R}^{-1}\tilde{B}^\top\tilde{A}^{-\top})\cup\{0\}
=
\mathrm{spec}(\tilde{Q}^{\frac{1}{2}}\tilde{A}^{-1}\tilde{B}\tilde{R}^{-1}\tilde{B}^\top\tilde{A}^{-\top}
\tilde{Q}^{\frac{1}{2}})\cup\{0\} \subset\mathbb{R}_{\geq 0},$$
whereby $D$ is invertible. In particular, the $22$-block of the
inverse of the left-hand side of (\ref{eq:Qinv}) is given by
$-(I+\tilde{Q}\tilde{A}^{-1}\tilde{B}\tilde{R}^{-1}\tilde{B}^\top
\tilde{A}^{-\top})^{-1}\tilde{Q} =
-\tilde{Q}^{\frac{1}{2}}(I+\tilde{Q}^{\frac{1}{2}}\tilde{A}^{-1}\tilde{B}\tilde{R}^{-1}\tilde{B}^\top
\tilde{A}^{-\top}Q^{\frac{1}{2}})^{-1}\tilde{Q}^{\frac{1}{2}}$,
which is congruent to the $33$-block of $D^{-1}$ via the
transformation $\tilde{A}^{-1}$ in view of (\ref{eq:QQinv}), as
claimed.
\end{proof}

\subsection{Proof of Lemma~\ref{lem:solved}}
\begin{proof}
  Recall that $\hat{R}_j\succ 0$, $\hat{Q}_j\succeq 0$,
  $\hat{Q}_j-\hat{S}_j^\top \hat{R}^{-1}\hat{S}_j \succeq 0$ and
  $\tilde{A}_j = \hat{A}_j+\hat{B}_j\hat{R}_j^{-1}\hat{S}_j$ is
  non-singular; n.b., the structure of $\tilde{A}_j$ is the same as
  the block bi-diagonal structure of $\hat{A}_j$ with identity matrices
  along the block diagonal. Now using Lemma~\ref{lem:QRS}, observe
  that $D_j$ is invertible for $j=1,\ldots,N$.
  Given the structure of $\Upsilon_j$, the matrix $\Sigma_j$ is the
  same as $D_j$ except for the $11$-block, which is $\hat{Q}_j$ in the
  case of the latter and
  $\check{Q}_j=\hat{Q}_j - \hat{E}_j^\top (\Sigma_{j+1}^{-1})_{33}
  \hat{E}_j$
  in the former. By Lemma~\ref{lem:QRS}, $(D_N^{-1})_{33} \preceq 0$,
  so that $(\Sigma_{j+1}^{-1})_{33} \preceq 0$, whereby
  $\check{Q}_j \succeq 0$, and thus, $\Sigma_j$ is non-singular for
  $j=N-1$ by Lemma~\ref{lem:QRS} again. Continuing this argument
  inductively yields the invertibility of $\Sigma_j$ for
  $j=N-2,\ldots, 1$.  As such, the specified recursions for
  $\Sigma_j$, $\tilde{\rho}_j$ and $\delta_j$ are all well defined.

  With $\Sigma_N=D_N$, solving the last block of~\eqref{eq:linearized}
  gives
  $\delta_N=\Sigma_N^{-1}\rho_N+\Sigma_N^{-1}\Upsilon_N\delta_{N-1}$.
  In turn,
\begin{align*}
\begin{bmatrix}
D_1 & -\Upsilon_2^\top & 0 & \ldots & 0 \\
-\Upsilon_2 & D_2 & -\Upsilon_3^\top & \ddots & \vdots \\
0 & -\Upsilon_3 & \ddots & \ddots & 0 \\
\vdots & \ddots    & \ddots &  D_{N-2}  & -\Upsilon_{N-1}^\top \\
0 & \cdots & 0 & -\Upsilon_{N-1} & \Sigma_{N-1}
\end{bmatrix}
 \begin{bmatrix}
\delta_1 \\ \delta_2 \\ 
\vdots \\  
\vdots \\
\delta_{N-1} 
\end{bmatrix} = \begin{bmatrix}
\rho_1 \\ \rho_2 \\ 
\vdots  \\ 
\vdots \\
\tilde{\rho}_{N-1}
\end{bmatrix},
\end{align*}
where $\Sigma_{N-1}$ and $\tilde{\rho}_{N-1}$ are defined as in the
statement of the lemma.  Continuing in this fashion yields the
remaining expressions for $\Sigma_j$, $\tilde{\rho}_j$ and
$\delta_j$. The first block equation eventually becomes
$\Sigma_1 \delta_1^k=\tilde{\rho}_1$, and thus,
$\delta_1=\Sigma_1^{-1}\tilde{\rho}_1$.
\end{proof}

\subsection{Proof of Lemma~\ref{lem:invertSigma}}

\begin{proof} 
  It is shown in the proof of Lemma~\ref{lem:solved} that
  $-\hat{E}_{j+1}^\top(\Sigma_{j+1}^{-1})_{33}\hat{E}_{j+1} \succeq
  0$. As such,
 $$
\left[\begin{matrix} \hat{Q}_j & \hat{S}_j^\top \\ \hat{S}_j &
    \hat{R}_j
 \end{matrix}\right] + 
\left[\begin{matrix}
    -\hat{E}_{j+1}^\top(\Sigma_{j+1}^{-1})_{33}\hat{E}_{j+1} & 0 \\ 0 & 0
 \end{matrix}\right] \succeq 0,$$ and thus, $\tilde{Q}_j =
\bar{Q}_j - \tilde{S}_j^\top \tilde{R}^{-1} \tilde{S} \succeq 0$
as claimed. Also observe that $\tilde{A}_j$ is non-singular because
it has the same lower block bi-diagonal structure as $\hat{A}_j$ with
identity matrices along the block diagonal. Now recall that
$\Theta_j\succ 0$, whereby
$\left[\begin{smallmatrix} 0 & I \\ \Theta_j &
    \Lambda_j\end{smallmatrix} \right]^{-1}=\left[\begin{smallmatrix}
    -\Theta^{-1}\!\Lambda_j & ~\Theta^{-1} \\ I & 0 \end{smallmatrix}
\right]$.  Using this and the structure of (\ref{eq:tosolve}) yields
\begin{align*}
\begin{bmatrix}
  X_4 \\ X_5
\end{bmatrix}
=
\begin{bmatrix} 
0 & I \\ \Theta_j & \Lambda_j
\end{bmatrix}^{-1}
\left(
\begin{bmatrix}
Y_4 \\ Y_5
\end{bmatrix}
-
\begin{bmatrix}
\hat{M}_j & \hat{L}_j & 0 \\
0 & 0 & 0 
\end{bmatrix}
\begin{bmatrix}
X_1 \\ X_2 \\ X_3
\end{bmatrix}
\right)
= \begin{bmatrix} 
0 & I \\ \Theta_j & \Lambda_j
\end{bmatrix}^{-1}
\begin{bmatrix}
Y_4 \\ Y_5
\end{bmatrix}
+
\Theta_j^{-1}\Lambda_j
\begin{bmatrix}
\hat{M}_j & \hat{L}_j 
\end{bmatrix} 
\begin{bmatrix}
X_1 \\ X_2
\end{bmatrix}
\end{align*}
and
\begin{align}
\label{eq:soltop}
\begin{bmatrix}
X_1 \\ X_2 \\ X_3
\end{bmatrix}
=
\begin{bmatrix} 
\bar{Q}_j & \tilde{S}_j^{\top} & -\hat{A}_j^\top \\
\tilde{S}_j & \tilde{R}_j & \hat{B}_j^\top \\
-\hat{A}_j & \hat{B}_j & 0 
\end{bmatrix}^{-1}
\begin{bmatrix}
\begin{bmatrix}
I & 0 & 0 \\
0 & I & 0 \\ 
0 & 0 & I 
\end{bmatrix}
&
\begin{bmatrix}  
-
\begin{bmatrix} \hat{M}_j^\top & 0 \\ \hat{L}_j^\top & 0  \end{bmatrix}
\begin{bmatrix} 0 & I \\ \Theta_j & \Lambda_j \end{bmatrix}^{-1} \\
\begin{bmatrix} 0 & 0 \end{bmatrix}
\end{bmatrix}
\end{bmatrix}
\begin{bmatrix}
  \begin{bmatrix} Y_1 \\ Y_2 \\ Y_3 \end{bmatrix} \\ \begin{bmatrix}
    Y_4 \\ Y_5 \end{bmatrix}
\end{bmatrix}.
\end{align}
In view of (\ref{eq:QQinv}),
\begin{align*}
\begin{bmatrix} 
    \bar{Q}_j & \tilde{S}_j^\top & -\hat{A}_j^\top \\ \tilde{S}_j &
    \tilde{R}_j & \hat{B}_j^\top 
    \\ -\hat{A}_j & \hat{B}_j & 0
\end{bmatrix}^{-1}
&\!\!\!\!=
\begin{bmatrix}
I & 0 & 0 \\ 
-\tilde{R}_j^{-1}\tilde{S}_j & I & -\tilde{R}_j^{-1}\tilde{B}_j^\top\tilde{A}_j^{-\top} \\
0 & 0 & \tilde{A}_j^{-\top}
\end{bmatrix}\\
&\quad\quad\quad\quad
\begin{bmatrix}
\tilde{Q}_j & 0 & -I \\ 0 & \tilde{R}_j & 0 \\ -I & 0 &
-\tilde{A}_j^{-1}\tilde{B}_j\tilde{R}_j^{-1}\tilde{B}_j^\top\tilde{A}_j^{-\top} 
\end{bmatrix}^{-1}
\begin{bmatrix}
I & -\tilde{S}_j^\top\tilde{R}_j^{-1} & 0 \\
0 & I & 0 \\
0 & -\tilde{A}_j^{-1}\tilde{B}_j\tilde{R}_j^{-1} & \tilde{A}_j^{-1}
\end{bmatrix}.
\end{align*}
As such, the result follows by noting that
\begin{align}
&
\begin{bmatrix}
\tilde{Q}_j & 0 & -I \\ 0 & \tilde{R}_j & 0 \\ -I & 0 &
-\tilde{A}_j^{-1}\tilde{B}_j\tilde{R}_j^{-1}\tilde{B}_j^\top\tilde{A}_j^{-\top} 
\end{bmatrix}^{-1}
\nonumber \\
&\quad\quad=
\begin{bmatrix}
  \tilde{A}_j^{-1}\tilde{B}_j\tilde{R}_j^{-1}\tilde{B}_j^\top\tilde{A}_j^{-\top}\tilde{W}_j
  & 0 &
  \tilde{A}_j^{-1}\tilde{B}_j\tilde{R}_j^{-1}\tilde{B}_j^\top\tilde{A}_j^{-\top}
  \tilde{W}_j\tilde{Q}_j
  -I\\
  0 & \tilde{R}^{-1} & 0 \\
  -\tilde{W}_j & 0 & -\tilde{W}_j\tilde{Q}_j
\end{bmatrix}, \label{eq:OmInv}
\end{align}
where
$\tilde{W}_j = (I + \tilde{Q}_j
\tilde{A}_j^{-1}\tilde{B}_j\tilde{R}_j^{-1}\tilde{B}_j^\top\tilde{A}_j^{-\top})^{-1}$.
Indeed, the matrix in (\ref{eq:OmInv}) is precisely $\Omega_j$ in
(\ref{eq:OmegaInv}). Application of the Sherman-Morrison-Woodbury
matrix inversion lemma gives the equivalent expression, including the
expression (\ref{eq:Wtil}) for $\tilde{W}_j$, and (\ref{eq:Ztil}) for
$\tilde{Z}_j =
\tilde{A}_j^{-1}\tilde{B}_j\tilde{R}_j^{-1}\tilde{B}_j^\top\tilde{A}_j^{-\top}\tilde{W}_j$.
\end{proof}

\bibliography{ijc2016}

\end{document}